\documentclass[11pt]{article}
\usepackage[top=2cm,bottom=2cm,left=2cm,right=2cm]{geometry}
\usepackage{amsmath,amsthm,amssymb}
\usepackage{graphicx,enumerate,url,color}

\def\ep{\varepsilon}

\def\G{\Gamma}

\def\cut{\setminus}

\def\C{\mathcal C}
\def\re{\mbox{Re}}

\numberwithin{equation}{section}
\newtheorem{thm}{Theorem}[section]

\begin{document}

\title{Plancherel-Rotach Asymptotics of Second-Order Difference Equations with Linear Coefficients}
\author{Xiang-Sheng Wang\footnote{Email: xswang@semo.edu}\\
Department of Mathematics\\Southeast Missouri State University\\Cape Girardeau, Missouri 63701}

%\date{}
\maketitle

\begin{abstract}
In this paper, we provide a complete Plancherel-Rotach asymptotic analysis of polynomials that satisfy a second-order difference equation with linear coefficients. According to the signs of the parameters, we classify the difference equations into six cases and derive explicit asymptotic formulas of the polynomials in the outer and oscillatory regions, respectively. It is remarkable that the zero distributions of the polynomials may locate on the imaginary line or even on a sideways Y-shape curve in some cases.
\end{abstract}



%%%%%%%%%%%%%%%%%%%%%%%%%%%%%%%%%%%%%%%%%%%%%%%%%%%%%%%%%%%
\section{Introduction}
All of the classical hypergeometric (monic) orthogonal polynomials $\pi_n(x)$ within Askey scheme \cite{KS98} satisfy the following second-order linear difference equation
\begin{align}
  \pi_{n+1}(x)=(x-A_n)\pi_n(x)-B_n\pi_{n-1}(x),~~\pi_0(x)=1,~\pi_1(x)=x-A_0,
\end{align}
where the coefficients $A_n$ and $B_n$ are polynomials or rational functions of $n$.
For instance, the Charlier polynomials correspond to $A_n=n+a$ and $B_n=na$; the Hermite polynomials correspond to $A_n=0$ and $B_n=n/2$; and the Chebyshev polynomials correspond to $A_n=0$ and $B_n=1/4$.
In this paper, we will provide a complete Plancherel-Rotach asymptotic analysis of second-order difference equations with linear coefficients, namely, $A_n$ and $B_n$ are linear functions of $n$. Upon a shift on $x$, we may assume $A_n=dn$ and $B_n=an+b$.

There are plenty of methods developed for asymptotic analysis of orthogonal polynomials: if the polynomials can be expressed in terms of an integral, one may adopt the classical Laplace's method and steepest-descent method \cite{Wo89}; if the polynomials satisfy a second-order linear differential equation, the well-known WKB method \cite{Ol97} can be applied; if the polynomials have an explicit orthogonal weight with certain nice properties, we may use the Riemann-Hilbert approach and Deift-Zhou nonlinear steepest-descent method \cite{BKMM07,DKMVZ99,DZ93}. However, few studies in the previous literature were considering asymptotic analysis of polynomials via difference equations due to the loss of continuity. Van Assche and Geronimo \cite{VAG89} did some pioneer works in this field and obtained asymptotic formulas in the outer region, where trapezoidal rule was used to build a bridge from discreteness to continuity. Wong and Li \cite{WL92} derived two linearly independent solutions in the oscillatory region, while determining the coefficients of the linear combination of the two solutions with given initial values was left as an open problem. In a series of work \cite{Wa01,WW02,WW03,WW05}, Wang and Wong established a beautiful lemma on Airy functions to derive uniform asymptotic formulas near the turning points. It is noted that their results were based on the assumption that the asymptotic formulas in the oscillatory region were given. Recently, Wang and Wong \cite{WW12} completed this framework by introducing a matching method to determine the coefficients of linear combination of Wong-Li solutions in the oscillatory region from Van Assche-Geronimo solutions in the outer region. Therefore, a systematic method of asymptotic analysis on difference equations was formulated. This method was successively applied in the study of several indeterminate moment problems \cite{CL14,DIW13} where only difference equations were known and thus the classical Laplace's method, steepest descent method, WKB method, Riemann-Hilbert approach and Deift-Zhou nonlinear steepest-descent method seem to be unapplicable.

To further develop the difference equation technique, we study a general second-order linear difference equation with linear coefficients. We are interested in the Plancherel-Rotach asymptotic formulas of solutions in the outer region and oscillatory region. According to the signs of the parameters $d$ and $a$, we classify the equations into six cases: I.A) $d>0$ and $a>0$; I.B) $d>0$ and $a<0$; I.C) $d>0$ and $a=0$; II.A) $d=0$ and $a>0$; II.B) $d=0$ and $a<0$; II.C) $d=0$ and $a=0$. The cases with $d<0$ can be transformed to the cases with $d>0$ by a simple reflection. Note that the classical orthogonal polynomials (Charlier, Hermite and Chebyshev, for instance) always have nonnegative $a\ge0$ and their zeros are always real. However, if we choose $a<0$, as we shall see later, the zero distributions of the polynomials $\pi_n(x)$ may lie on the imaginary line (subcase II.B) or even on a sideways Y-shape curve (subcase I.B).

The rest of this paper is organized as follows. In Section 2, we focus on the case $d\neq0$ and divide this case into three subcases according to the sign of $a$. In Section 3, we investigate the special case $d=0$ and again consider three subcases $a>0$, $a<0$ and $a=0$ in three subsections, respectively.

%%%%%%%%%%%%%%%%%%%%%%%%%%%%%%%%%%%%%%%%%%%%%%%%%%%%%%%%%%%%
\section{Case I: $d\neq0$}
Upon a transformation $x\to-x$ and $\pi_n\to(-1)^n\pi_n$, we may assume without loss of generality that $d>0$.
In the following three subsections, we shall consider three subcases $a>0$, $a<0$ and $a=0$, respectively.
\subsection{Subcase I.A: $a>0$}
We first state our theorem.
\begin{thm}\label{IA}
  Assume $d>0$ and $a>0$. Let $x=ny$ and $y=d+z/\sqrt n$. As $n\to\infty$, for $z\in\C\cut[-\sqrt n d,2\sqrt a]$, we have
\begin{align}\label{IA-1}
  \pi_n(nd+\sqrt nz)\sim&(n/e)^n({z+\sqrt{z^2-4a}\over2\sqrt n})^n({z+\sqrt{z^2-4a}\over2(\sqrt n d+z)})^{-a/d^2-\sqrt n(\sqrt n d+z)/d}
  \nonumber\\&~~\times({\sqrt n d+z\over\sqrt{z^2-4a}})^{1/2}\times\exp[{2a-z^2-4\sqrt n dz+(z+4\sqrt n d)\sqrt{z^2-4a}\over4d^2}];
\end{align}
and for $z$ in a neighborhood of $(-2\sqrt a,2\sqrt a)$, we have
\begin{align}\label{IA-2}
  \pi_n(nd+\sqrt nz)\sim&(n/e)^n({\sqrt{a}\over\sqrt n})^{-a/d^2-\sqrt n z/d}
  (d+z/\sqrt n)^{a/d^2+\sqrt n (\sqrt n d+z)/d}({\sqrt n d+z\over\sqrt{4a-z^2}})^{1/2}\exp[{2a-z^2-4\sqrt n dz\over4d^2}]
  \nonumber\\&~~\times  2\cos[(n-{a\over d^2}-{\sqrt n (\sqrt n d+z)\over d})\arccos{z\over2\sqrt a}-{\pi\over4}+{(z+4\sqrt n d)\sqrt{4a-z^2}\over4d^2}];
\end{align}
and for $z$ in a neighborhood of $(-\sqrt n d,-2\sqrt a)$, we have
\begin{align}\label{IA-3}
  \pi_n(nd+\sqrt nz)\sim&(n/e)^n({-z+\sqrt{-z-2\sqrt a}\sqrt{-z+2\sqrt a}\over2\sqrt n})^{-a/d^2-\sqrt n z/d}
  \nonumber\\&~~\times(d+z/\sqrt n)^{a/d^2+\sqrt n (\sqrt n d+z)/d}({\sqrt n d+z\over\sqrt{-z-2\sqrt a}\sqrt{-z+2\sqrt a}})^{1/2}
  \nonumber\\&~~\times\exp[{2a-z^2-4\sqrt n dz-(z+4\sqrt n d)\sqrt{-z-2\sqrt a}\sqrt{-z+2\sqrt a}\over4d^2}]
  \nonumber\\&~~\times2\cos[\pi(-a/d^2-\sqrt n z/d-1/2)].
\end{align}
\end{thm}
\begin{proof}
Denote
$$\pi_n(x)=\Pi_{k=1}^n w_k(x).$$
It follows that
$$w_{k+1}(x)=x-dk-{ak+b\over w_k(x)}.$$
Let $x=ny$ with $y\in\C\cut[0,d+2\sqrt a/\sqrt n]$.
We have as $n\to\infty$,
\begin{align*}
  w_k(x)\sim&{x-dk+\sqrt{(x-dk)^2-4ak}\over 2}
  \times\left\{1+{d\over2\sqrt{(x-dk)^2-4ak}}+{dx-d^2k\over2[(x-dk)^2-4ak]}\right\}.
\end{align*}
The above asymptotic formula can be obtained by successive approximation and proved rigorously by induction on $k$.
Since $(x-dk)^2-4ak$ is of order $O(n)$ for any $k=1,\cdots,n$, we have
\begin{align*}
  \ln\pi_n\sim&\sum_{k=1}^n\left\{\ln{x-dk+\sqrt{(x-dk)^2-4ak}\over 2}+{d\over2\sqrt{(x-dk)^2-4ak}}+{dx-d^2k\over2[(x-dk)^2-4ak]}\right\}.
\end{align*}
We will use trapezoidal rule to approximate the three summations on the right-hand side of the above formula.
Firstly, we obtain
\begin{align*}
  &\sum_{k=1}^n\ln{x-dk+\sqrt{(x-dk)^2-4ak}\over 2}=\sum_{k=1}^n\ln{ny-dk+\sqrt{(ny-dk)^2-4ak}\over 2}
  \\\sim& n\ln{n\over 2}+n\int_0^1\ln[y-dt+\sqrt{(y-dt)^2-4at/n}]dt
  +{1\over2}\ln{y-d+\sqrt{(y-d)^2-4a/n}\over2y}.
\end{align*}
A simple integration gives
\begin{align*}
  &n\int_0^1\ln[y-dt+\sqrt{(y-dt)^2-4at/n}]dt
  \\\sim& n\bigg\{t\ln[y-dt+\sqrt{(y-dt)^2-4at/n}+{\sqrt{(y-dt)^2-4at/n}\over2d}-{t\over2}
  \\&~~-({a\over nd^2}+{y\over d})\ln[dy+2a/n-d^2t+d\sqrt{(y-dt)^2-4at/n}]\bigg\}\bigg|_0^1
  \\\sim& n\ln[y-d+\sqrt{(y-d)^2-4a/n}]+{n\over2d}(\sqrt{(y-d)^2-4a/n}-y)-{n\over2}
  \\&~~-({a\over d^2}+{ny\over d})\ln{y-d+\sqrt{(y-d)^2-4a/n}+2a/(nd)\over2y+2a/(nd)}.
\end{align*}
For the sake of convenience, we introduce a new scale: $y=d+z/\sqrt n$ with $z\in\C\cut[-\sqrt n d,2\sqrt a]$.
It follows from the above two formulas that
\begin{align*}
  &\sum_{k=1}^n\ln{x-dk+\sqrt{(x-dk)^2-4ak}\over 2}
  \sim n\ln{n\over 2}+n\ln{z+\sqrt{z^2-4a}\over\sqrt n}+{\sqrt n\over2d}(\sqrt{z^2-4a}-z)-n
  \\&~~~~~~-({a\over d^2}+{\sqrt n(\sqrt n d+z)\over d})\ln{z+\sqrt{z^2-4a}+2a/(\sqrt n d)\over2(\sqrt n d+z+a/(\sqrt n d)}
  +{1\over2}\ln{z+\sqrt{z^2-4a}\over2(\sqrt n d+z)}.
\end{align*}
A further application of trapezoidal rule yields
\begin{align*}
  \sum_{k=1}^n\left\{{d\over2\sqrt{(x-dk)^2-4ak}}+{dx-d^2k\over2[(x-dk)^2-4ak]}\right\}
  \sim&\int_0^1{d\over2\sqrt{(y-dt)^2-4at/n}}+{dy-d^2t\over2[(y-dt)^2-4at/n]}dt
  \\\sim&{1\over2}\ln{2(\sqrt n d+z)\over z+\sqrt{z^2-4a}}+{1\over2}\ln{\sqrt n d+z\over\sqrt{z^2-4a}}.
\end{align*}
Adding the above two formulas gives
\begin{align*}
  \ln\pi_n\sim& n\ln{n\over 2}+n\ln{z+\sqrt{z^2-4a}\over\sqrt n}+{\sqrt n\over2d}(\sqrt{z^2-4a}-z)-n
  \\&~~-({a\over d^2}+{\sqrt n(\sqrt n d+z)\over d})\ln{z+\sqrt{z^2-4a}+2a/(\sqrt n d)\over2(\sqrt n d+z)+2a/(\sqrt n d)}
  +{1\over2}\ln{\sqrt n d+z\over\sqrt{z^2-4a}}.
\end{align*}
Since
\begin{align*}
  \ln{z+\sqrt{z^2-4a}+2a/(\sqrt n d)\over2(\sqrt n d+z)+2a/(\sqrt n d)}&\sim
  \ln{z+\sqrt{z^2-4a}\over2(\sqrt n d+z)}
  +{2a\over(\sqrt n d)(z+\sqrt{z^2-4a})}
  \\&~~-{2a^2\over d^2n(z+\sqrt{z^2-4a})^2}-{a\over(\sqrt n d)(\sqrt n d+z)},
\end{align*}
we have
\begin{align*}
  \ln\pi_n\sim& n\ln n+n\ln{z+\sqrt{z^2-4a}\over2\sqrt n}-n
  -({a\over d^2}+{\sqrt n(\sqrt n d+z)\over d})\ln{z+\sqrt{z^2-4a}\over2(\sqrt n d+z)}
  \\&~~+{\sqrt n(\sqrt{z^2-4a}-z)\over2d}-[{2a(\sqrt n d+z)\over(d^2)(z+\sqrt{z^2-4a})}-{2a^2\over d^2(z+\sqrt{z^2-4a})^2}-{a\over(d^2)}]
  +{1\over2}\ln{\sqrt n d+z\over\sqrt{z^2-4a}}.
\end{align*}
A simple calculation yields
\begin{align*}
  &{\sqrt n(\sqrt{z^2-4a}-z)\over2d}-[{2a(\sqrt n d+z)\over(d^2)(z+\sqrt{z^2-4a})}-{2a^2\over d^2(z+\sqrt{z^2-4a})^2}-{a\over(d^2)}]
  \\=&{\sqrt n d(\sqrt{z^2-4a}-z)\over2d^2}-[{(\sqrt n d+z)(z-\sqrt{z^2-4a})\over(2d^2)}-{(z-\sqrt{z^2-4a})^2\over 8d^2}-{a\over(d^2)}]
  \\=&{2\sqrt n d(\sqrt{z^2-4a}-z)\over4d^2}-[{2(\sqrt n d+z)(z-\sqrt{z^2-4a})\over(4d^2)}-{(z^2-2a-z\sqrt{z^2-4a})\over 4d^2}-{4a\over4d^2}]
  \\=&{-z^2+2a-4\sqrt n dz+(z+4\sqrt n d)\sqrt{z^2-4a}\over4d^2}.
\end{align*}
Consequently,
\begin{align*}
  \ln\pi_n\sim&  n\ln n-n+n\ln{z+\sqrt{z^2-4a}\over2\sqrt n}
  -({a\over d^2}+{\sqrt n(\sqrt n d+z)\over d})\ln{z+\sqrt{z^2-4a}\over2(\sqrt n d+z)}
  \\&~~+{-z^2+2a-4\sqrt n dz+(z+4\sqrt n d)\sqrt{z^2-4a}\over4d^2}
  +{1\over2}\ln{\sqrt n d+z\over\sqrt{z^2-4a}}.
\end{align*}
Recall that $x=ny$ and $y=d+z/\sqrt n$. For any $z\in\C\cut[-\sqrt n d,2\sqrt a]$, we have $\pi_n(nd+\sqrt nz)\sim\Phi_n(z)$ as $n\to\infty$, where
\begin{align*}
  \Phi_n(z):=&(n/e)^n({z+\sqrt{z^2-4a}\over2\sqrt n})^n({z+\sqrt{z^2-4a}\over2(\sqrt n d+z)})^{-a/d^2-\sqrt n(\sqrt n d+z)/d}
  \\&~~\times({\sqrt n d+z\over\sqrt{z^2-4a}})^{1/2}\times\exp[{2a-z^2-4\sqrt n dz+(z+4\sqrt n d)\sqrt{z^2-4a}\over4d^2}].
\end{align*}
This proves \eqref{IA-1}.
Note that $\Phi_n(z)$ has a branch cut on $[-\sqrt n d,2\sqrt a]$. We take the one-sided limits and define
\begin{align*}
  \Phi_n^\pm(z):=\lim_{\ep\to0^+}\Phi_n(z\pm i\ep),~~z\in(-\sqrt n d,2\sqrt a).
\end{align*}
It is readily seen that $\Phi_n^\pm(z)$ can be analytically extended to a neighborhood of $(-\sqrt n d,2\sqrt a)$.
Moreover, if $z=z_1+iz_2$ with $z_1\in(-\sqrt n d,2\sqrt a)$ and $z_2>0$, then $\Phi_n(z)=\Phi_n^+(z)$ and $\Phi_n^-(z)/\Phi_n^+(z)$ is exponentially small as $n\to\infty$.
On the other hand, if $z=z_1+iz_2$ with $z_1\in(-\sqrt n d,2\sqrt a)$ and $z_2<0$, then $\Phi_n(z)=\Phi_n^-(z)$ and $\Phi_n^+(z)/\Phi_n^-(z)$ is exponentially small as $n\to\infty$.
It follows that as $n\to\infty$, $\pi_n(nd+\sqrt nz)\sim\Phi_n(z)\sim\Phi_n^+(z)+\Phi_n^-(z)$ for all $z=z_1+iz_2$ with $z_1\in(-\sqrt n d,2\sqrt a)$ and $z_2\neq0$. By analytically continuity, we obtain $\pi_n(nd+\sqrt nz)\sim\Phi_n^+(z)+\Phi_n^-(z)$ for $z$ in a neighborhood of $(-\sqrt n d,2\sqrt a)$.
A simple calculation gives
\begin{align*}
  \Phi_n^+(z)+\Phi_n^-(z)=&(n/e)^n({\sqrt{a}\over\sqrt n})^{-a/d^2-\sqrt n z/d}
  \\&~~\times(d+z/\sqrt n)^{a/d^2+\sqrt n (\sqrt n d+z)/d}({\sqrt n d+z\over\sqrt{4a-z^2}})^{1/2}\exp[{2a-z^2-4\sqrt n dz\over4d^2}]
  \\&~~\times  2\cos[(n-{a\over d^2}-{\sqrt n (\sqrt n d+z)\over d})\arccos{z\over2\sqrt a}-{\pi\over4}+{(z+4\sqrt n d)\sqrt{4a-z^2}\over4d^2}]
\end{align*}
for $z$ in a neighborhood of $(-2\sqrt a,2\sqrt a)$, and
\begin{align*}
  \Phi_n^+(z)+\Phi_n^-(z)=&(n/e)^n({-z+\sqrt{-z-2\sqrt a}\sqrt{-z+2\sqrt a}\over2\sqrt n})^{-a/d^2-\sqrt n z/d}
  \\&~~\times(d+z/\sqrt n)^{a/d^2+\sqrt n (\sqrt n d+z)/d}({\sqrt n d+z\over\sqrt{-z-2\sqrt a}\sqrt{-z+2\sqrt a}})^{1/2}
  \\&~~\times\exp[{2a-z^2-4\sqrt n dz-(z+4\sqrt n d)\sqrt{-z-2\sqrt a}\sqrt{-z+2\sqrt a}\over4d^2}]
  \\&~~\times2\cos[\pi(-a/d^2-\sqrt n z/d-1/2)]
\end{align*}
for $z$ in a neighborhood of $(-\sqrt n d,-2\sqrt a)$. This completes the proof of \eqref{IA-2} and \eqref{IA-3}.
\end{proof}

\subsection{Case I.B: $a<0$}
For the case $a<0$, we observe from numerical simulation that the zeros of $\pi_n$ are not solely lying on the real line, instead, they will locate on a sideways Y-shape curve (cf. Figure \ref{fig-Y-shape}). This will be theoretically justified in the following theorem.
\begin{thm}
Assume $d>0$ and $a<0$. Let $x=ny$ and $y=d+z/\sqrt n$. Denote $A=-a>0$.
Let $\G_A$ be the curve in the left-half complex plane defined by the following equation
\begin{equation}\label{GA}
  \re\left\{2\sqrt{z-2i\sqrt{A}}\sqrt{z+2i\sqrt{A}}-z\ln{z+\sqrt{z-2i\sqrt{A}}\sqrt{z+2i\sqrt{A}}\over z-\sqrt{z-2i\sqrt{A}}\sqrt{z+2i\sqrt{A}}}\right\}=0.
\end{equation}
It is noted that the above equation formulates a sideways V-shape curve that is symmetric about the x-axis with two end points $\pm2i\sqrt A$; see Figure \ref{fig-Y-shape}.
Let $z_A<0$ be the intersection of $\G_A$ with the negative real line. To be specific, $z_A$ is the negative real root of the following equation
\begin{align}\label{zA}
  2\sqrt{z_A^2+4A}-z_A\ln{z_A+\sqrt{z_A^2+4A}\over-z_A+\sqrt{z_A^2+4A}}=0.
\end{align}
As $n\to\infty$, we have for $z\in\C\cut([-\sqrt n d,z_A]\cup\G_A)$,
\begin{align}\label{IB-1}
  \pi_n(nd+\sqrt nz)\sim&(n/e)^n({z+\sqrt{z-2i\sqrt{A}}\sqrt{z+2i\sqrt{A}}\over2\sqrt n})^{A/d^2-\sqrt n z/d}
  \nonumber\\&~~\times(d+z/\sqrt n)^{-A/d^2+\sqrt n (\sqrt n d+z)/d}({\sqrt n d+z\over\sqrt{z-2i\sqrt{A}}\sqrt{z+2i\sqrt{A}}})^{1/2}
  \nonumber\\&~~\times\exp[{-2A-z^2-4\sqrt n dz+(z+4\sqrt n d)\sqrt{z-2i\sqrt{A}}\sqrt{z+2i\sqrt{A}}\over4d^2}];
\end{align}
and for $z$ in a neighborhood of $(-\sqrt n d,z_A)$, we have
\begin{align}\label{IB-2}
  \pi_n(nd+\sqrt nz)\sim&(n/e)^n({-z+\sqrt{-z-2i\sqrt{A}}\sqrt{-z+2i\sqrt{A}}\over2\sqrt n})^{A/d^2-\sqrt n z/d}
  \nonumber\\&~~\times(d+z/\sqrt n)^{-A/d^2+\sqrt n (\sqrt n d+z)/d}({\sqrt n d+z\over\sqrt{-z-2i\sqrt{A}}\sqrt{-z+2i\sqrt{A}}})^{1/2}
  \nonumber\\&~~\times\exp[{-2A-z^2-4\sqrt n dz-(z+4\sqrt n d)\sqrt{-z-2i\sqrt{A}}\sqrt{-z+2i\sqrt{A}}\over4d^2}]
  \nonumber\\&~~\times2\cos[\pi[A/d^2-\sqrt n z/d-1/2]];
\end{align}
and for $z$ in a neighborhood of $\mathring\G_A:=\G_A\cut\{z_A,\pm 2i\sqrt{A}\}$, we have
\begin{align}\label{IB-3}
  &\pi_n(nd+\sqrt nz)\sim(\sqrt n/e)^n(\sqrt nd+z)^{-A/d^2+\sqrt n (\sqrt n d+z)/d+1/2}
  \times\exp[{-2A-z^2-4\sqrt n dz\over4d^2}]
  \nonumber\\&~~\times\{{[(z+\sqrt{z-2i\sqrt{A}}\sqrt{z+2i\sqrt{A}})/2]^{A/d^2-\sqrt n z/d}
  \over(\sqrt{z-2i\sqrt{A}}\sqrt{z+2i\sqrt{A}})^{1/2}}
  \exp[{(z+4\sqrt n d)\sqrt{z-2i\sqrt{A}}\sqrt{z+2i\sqrt{A}}\over4d^2}]
  \nonumber\\&~~~~~+{[(z-\sqrt{z-2i\sqrt{A}}\sqrt{z+2i\sqrt{A}})/2]^{A/d^2-\sqrt n z/d}
  \over(-\sqrt{z-2i\sqrt{A}}\sqrt{z+2i\sqrt{A}})^{1/2}}
  \exp[{-(z+4\sqrt n d)\sqrt{z-2i\sqrt{A}}\sqrt{z+2i\sqrt{A}}\over4d^2}]\}.
\end{align}
\end{thm}

\begin{figure}\label{fig-Y-shape}
\centering
\includegraphics[width=0.7\linewidth,height=8cm]{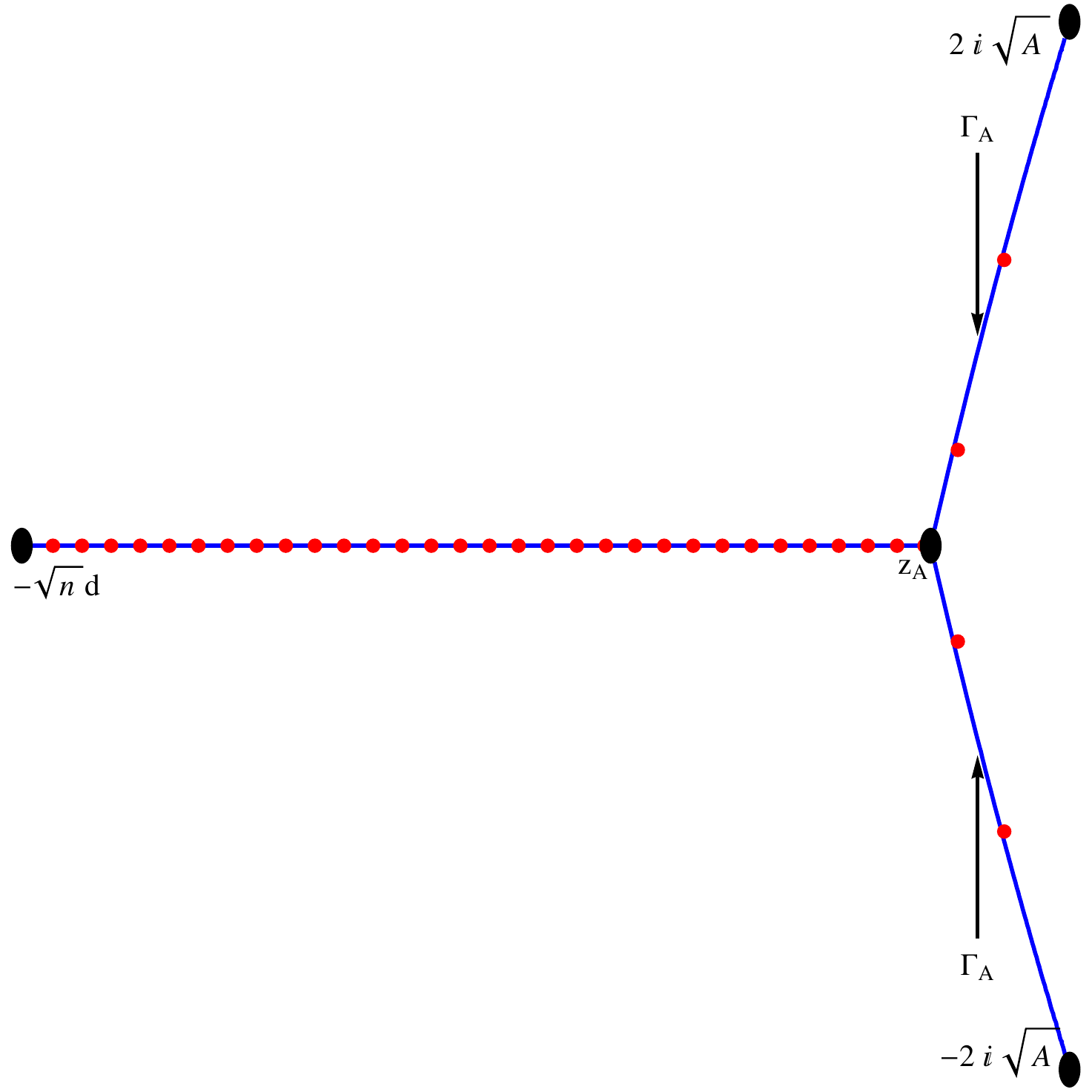}
\caption{The sideways Y-shape branch cut (curve) and zero distribution (dots).}
\end{figure}

\begin{proof}
  Similar to the proof of Theorem \ref{IA}, we obtain for $z\in\C\cut([-\sqrt n d,z_A]\cup\G_A)$,
\begin{align*}
  \pi_n(nd+\sqrt nz)\sim&(n/e)^n({z+\sqrt{z^2-4a}\over2\sqrt n})^{-a/d^2-\sqrt n z/d}
  \times(d+z/\sqrt n)^{a/d^2+\sqrt n (\sqrt n d+z)/d}({\sqrt n d+z\over\sqrt{z^2-4a}})^{1/2}
  \\&~~\times\exp[{2a-z^2-4\sqrt n dz+(z+4\sqrt n d)\sqrt{z^2-4a}\over4d^2}]
  \\\sim&(n/e)^n({z+\sqrt{z-2i\sqrt{A}}\sqrt{z+2i\sqrt{A}}\over2\sqrt n})^{A/d^2-\sqrt n z/d}
  \\&~~\times(d+z/\sqrt n)^{-A/d^2+\sqrt n (\sqrt n d+z)/d}({\sqrt n d+z\over\sqrt{z-2i\sqrt{A}}\sqrt{z+2i\sqrt{A}}})^{1/2}
  \\&~~\times\exp[{-2A-z^2-4\sqrt n dz+(z+4\sqrt n d)\sqrt{z-2i\sqrt{A}}\sqrt{z+2i\sqrt{A}}\over4d^2}].
\end{align*}
This gives \eqref{IB-1}. Denote the right-hand side of \eqref{IB-1} by $\Phi_n(z)$.
Note that $\Phi_n(z)$ is analytic on the complex plane except for a Y-shape branch cut $[-\sqrt n d,z_A]\cup\G_A$ that connects $-\sqrt nd$ and $\pm2i\sqrt A$.
Moreover, $\Phi_n(z)$ is one-side continuous on the branch cut.
Therefore, the functions
$$\Phi_n^\pm(z):=\lim_{\ep\to0^+}\Phi_n(z+i\ep)$$ are analytic in a neighborhood of the branch cut.
Note that for $z$ in a neighborhood of $(-\infty,z_A)$,
\begin{align*}
  {\Phi_n^+(z)\over\Phi_n^-(z)}&=\exp[2\pi i(A/d^2-\sqrt n z/d-1/2)];
\end{align*}
and for $z$ in a neighborhood of $\mathring\G_A$,
\begin{align*}
  {\Phi_n^+(z)\over\Phi_n^-(z)}&=-i({z+\sqrt{z-2i\sqrt{A}}\sqrt{z+2i\sqrt{A}}\over z-\sqrt{z-2i\sqrt{A}}\sqrt{z+2i\sqrt{A}}})^{A/d^2-\sqrt n z/d}
  \exp({z+4\sqrt n d\over2d^2}\sqrt{z-2i\sqrt{A}}\sqrt{z+2i\sqrt{A}}).
\end{align*}
It follows from the definition of $\G_A$ in \eqref{GA} that the ratio $\Phi_n^+/\Phi_n^-$ is exponentially large on one side and exponentially small on the other side of the branch cut $[-\sqrt n d,z_A]\cup\G_A$.
Using a similar argument in the proof of Theorem \ref{IA}, we obtain
$\pi_n(nd+\sqrt nz)\sim[\Phi_n^+(z)+\Phi_n^-(z)]$ for $z$ in a neighborhood of $(-\sqrt n d,z_A)\cup\mathring\G_A$.
A simple calculation yields
\begin{align*}
  \Phi_n^+(z)+\Phi_n^-(z)&=(n/e)^n({-z+\sqrt{z-2i\sqrt{A}}\sqrt{z+2i\sqrt{A}}\over2\sqrt n})^{A/d^2-\sqrt n z/d}
  \nonumber\\&~~\times(d+z/\sqrt n)^{-A/d^2+\sqrt n (\sqrt n d+z)/d}({\sqrt n d+z\over\sqrt{z-2i\sqrt{A}}\sqrt{z+2i\sqrt{A}}})^{1/2}
  \nonumber\\&~~\times\exp[{-2A-z^2-4\sqrt n dz-(z+4\sqrt n d)\sqrt{z-2i\sqrt{A}}\sqrt{z+2i\sqrt{A}}\over4d^2}]
  \nonumber\\&~~\times2\cos[\pi[A/d^2-\sqrt n z/d-1/2]]
\end{align*}
for $z$ in a neighborhood of $(-\sqrt n d,z_A)$, and
\begin{align*}
  &\Phi_n^+(z)+\Phi_n^-(z)
  \\=&(\sqrt n/e)^n(\sqrt nd+z)^{-A/d^2+\sqrt n (\sqrt n d+z)/d+1/2}
  \times\exp[{-2A-z^2-4\sqrt n dz\over4d^2}]
  \nonumber\\&\times\{{[(z+\sqrt{z-2i\sqrt{A}}\sqrt{z+2i\sqrt{A}})/2]^{A/d^2-\sqrt n z/d}
  \over(\sqrt{z-2i\sqrt{A}}\sqrt{z+2i\sqrt{A}})^{1/2}}
  \exp[{(z+4\sqrt n d)\sqrt{z-2i\sqrt{A}}\sqrt{z+2i\sqrt{A}}\over4d^2}]
  \nonumber\\&~+{[(z-\sqrt{z-2i\sqrt{A}}\sqrt{z+2i\sqrt{A}})/2]^{A/d^2-\sqrt n z/d}
  \over(-\sqrt{z-2i\sqrt{A}}\sqrt{z+2i\sqrt{A}})^{1/2}}
  \exp[{-(z+4\sqrt n d)\sqrt{z-2i\sqrt{A}}\sqrt{z+2i\sqrt{A}}\over4d^2}]\}
\end{align*}
for $z$ in a neighborhood of $\mathring\G_A$. This proves \eqref{IB-2} and \eqref{IB-3}.
\end{proof}

\subsection{Case I.C: $a=0$}
\begin{thm}
Assume $d>0$ and $a=0$. Let $x=ny$, we have as $n\to\infty$, If $y$ is bounded away from $[0,d]$, we have
\begin{align}\label{IC-1}
\pi_n(ny)\sim (n/e)^n({y\over y-d})^{ny/d+1/2}(y-d)^n
\end{align}
for $y\in\C\cut[0,d]$; and
\begin{align}\label{IC-2}
  \pi_n(ny)\sim&(n/e)^n(d-y)^n({y\over d-y})^{ny/d+1/2}
  \times2\cos[\pi(n-ny/d-1/2)].
\end{align}
for $y$ in a neighborhood of $(0,d)$.
\end{thm}
\begin{proof}
  Setting $z=\sqrt n(y-d)$ in \eqref{IA-1} and taking limit $a\to0^+$ yields \eqref{IC-1}. A standard argument of analytical continuity as in the proof of Theorem \ref{IA} gives \eqref{IC-2}.
\end{proof}

%%%%%%%%%%%%%%%%%%%%%%%%%%%%%%%%%%%%%%%%%%%%%%%%%%%%%%%%%%%%
\section{Case II: $d=0$}
In this section, we consider the critical case $d=0$. Again, we investigate three subcases according to the sign of $a$.
\subsection{Case II.A: $a>0$}
\begin{thm}\label{IIA}
Assume $d=0$ and $a>0$. Let $x=\sqrt n y$. As $n\to\infty$, we have for $y\in\C\cut[-2\sqrt a,2\sqrt a]$,
\begin{align}\label{IIA-1}
  \pi_n(\sqrt n y)\sim&({n\over4e})^{n/2}(y+\sqrt{y^2-4a})^n({y+\sqrt{y^2-4a}\over2\sqrt{y^2-4a}})^{1/2}({y+\sqrt{y^2-4a}\over2y})^{b/a}
  \times\exp[{ny\over4a}(y-\sqrt{y^2-4a})];
\end{align}
and for $y$ in a neighborhood of $(0,2\sqrt a)$, we have
\begin{align}\label{IIA-2}
  \pi_n(\sqrt n y)\sim&({na\over e})^{n/2}({\sqrt a\over\sqrt{2\sqrt a-y}\sqrt{2\sqrt a+y}})^{1/2}({\sqrt a\over y})^{b/a}\times\exp[{ny^2\over4a}]
  \nonumber\\&~~\times2\cos[(n+1/2+b/a)\arccos{y\over2\sqrt a}-\pi/4-{ny\over4a}\sqrt{2\sqrt a-y}\sqrt{2\sqrt a+y}];
\end{align}
and for $y$ in a neighborhood of $(-2\sqrt a,0)$, we have
\begin{align}\label{IIA-3}
  \pi_n(\sqrt n y)\sim&({na\over e})^{n/2}(-1)^n({\sqrt a\over\sqrt{2\sqrt a-y}\sqrt{2\sqrt a+y}})^{1/2}({\sqrt a\over-y})^{b/a}\times\exp[{ny^2\over4a}]
  \nonumber\\&~~\times2\cos[(n+1/2+b/a)\arccos{-y\over2\sqrt a}-\pi/4+{ny\over4a}\sqrt{2\sqrt a-y}\sqrt{2\sqrt a+y}]
\end{align}
\end{thm}
\begin{proof}
Denote
$$\pi_n(x)=\Pi_{k=1}^n w_k(x).$$
It follows that
$$w_{k+1}(x)=x-{ak+b\over w_k(x)}.$$
Let $x=\sqrt ny$ with $y\in\C\cut[-2\sqrt a,2\sqrt a]$.
We have as $n\to\infty$,
\begin{align*}
  w_k(x)\sim&{x+\sqrt{x^2-4ak}\over 2}
  \times\left\{1+{a\over x^2-4ak}-{2b\over(x+\sqrt{x^2-4ak})\sqrt{x^2-4ak}}\right\}.
\end{align*}
By trapezoidal rule, we obtain
\begin{align*}
  \ln\pi_n\sim& n\ln(\sqrt n/2)+n\int_0^1\ln(y+\sqrt{y^2-4at})dt+{1\over2}\ln{y+\sqrt{y^2-4a}\over2y}
  \\&~~+\int_0^1{a\over y^2-4at}dt-{2b\over(y+\sqrt{y^2-4at})\sqrt{y^2-4at}}dt
  \\\sim& n\ln(\sqrt n/2)+n[t\ln(y+\sqrt{y^2-4at})-{y\over4a}\sqrt{y^2-4at}-{t\over2}]\bigg|_0^1
  \\&~~+{1\over2}\ln{y+\sqrt{y^2-4a}\over2y}+{1\over4}\ln{y^2\over y^2-4a}+{b\over a}\ln(y+\sqrt{y^2-4at})\bigg|_0^1
  \\\sim& n\ln(\sqrt n/2)+n\ln(y+\sqrt{y^2-4a})-{ny\over4a}(\sqrt{y^2-4a}-y)-{n\over2}
  \\&~~+{1\over2}\ln{y+\sqrt{y^2-4a}\over2y}+{1\over4}\ln{y^2\over y^2-4a}+{b\over a}\ln{y+\sqrt{y^2-4a}\over2y}.
\end{align*}
Recall that $x=ny$. We then obtain $\pi_n(ny)\sim\Phi_n(y)$, where
\begin{align*}
  \Phi_n(y):=&({n\over4e})^{n/2}(y+\sqrt{y^2-4a})^n({y+\sqrt{y^2-4a}\over2\sqrt{y^2-4a}})^{1/2}({y+\sqrt{y^2-4a}\over2y})^{b/a}
  \times\exp[{ny\over4a}(y-\sqrt{y^2-4a})].
\end{align*}
By a standard argument of analytical continuity, we obtain
$\pi_n(ny)\sim\Phi_n^+(y)+\phi_n^-(y)$ for $y$ in a neighborhood of $(-2\sqrt a,0)\cup(0,2\sqrt a)$,
where
$$\Phi_n^\pm(y):=\lim_{\ep\to0^+}\Phi_n(y+i\ep).$$
For $y$ in a neighborhood of $(0,2\sqrt a)$, a simple calculation gives
\begin{align*}
  \Phi_n^+(y)+\phi_n^-(y)=&({n\over4e})^{n/2}(2\sqrt a)^n({2\sqrt a\over2\sqrt{2\sqrt a-y}\sqrt{2\sqrt a+y}})^{1/2}({2\sqrt a\over2y})^{b/a}\times\exp[{ny\over4a}(y)]
  \\&~~\times2\cos[(n+1/2+b/a)\arccos{y\over2\sqrt a}-\pi/4-{ny\over4a}\sqrt{2\sqrt a-y}\sqrt{2\sqrt a+y}]
  \\\sim&({na\over e})^{n/2}({\sqrt a\over\sqrt{2\sqrt a-y}\sqrt{2\sqrt a+y}})^{1/2}({\sqrt a\over y})^{b/a}\times\exp[{ny^2\over4a}]
  \\&~~\times2\cos[(n+1/2+b/a)\arccos{y\over2\sqrt a}-\pi/4-{ny\over4a}\sqrt{2\sqrt a-y}\sqrt{2\sqrt a+y}].
\end{align*}
Thus, \eqref{IIA-2} follows.
Note that for $\re y<0$, we can write
\begin{align*}
  \Phi_n(y)=&({n\over4e})^{n/2}(-1)^n(-y+\sqrt{-y-2\sqrt a}\sqrt{-y+2\sqrt a})^n
  ({-y+\sqrt{-y-2\sqrt a}\sqrt{-y+2\sqrt a}\over2\sqrt{-y-2\sqrt a}\sqrt{-y+2\sqrt a}})^{1/2}
  \\&~~\times({-y+\sqrt{-y-2\sqrt a}\sqrt{-y+2\sqrt a}\over-2y})^{b/a}
  \times\exp[{ny\over4a}(y+\sqrt{-y-2\sqrt a}\sqrt{-y+2\sqrt a})].
\end{align*}
It follows that for $y$ in a neighborhood of $(-2\sqrt a,0)$,
\begin{align*}
  \Phi_n^+(y)+\phi_n^-(y)=&({n\over4e})^{n/2}(-1)^n(2\sqrt a)^n({2\sqrt a\over2\sqrt{2\sqrt a-y}\sqrt{2\sqrt a+y}})^{1/2}({2\sqrt a\over-2y})^{b/a}\times\exp[{ny\over4a}(y)]
  \\&~~\times2\cos[(n+1/2+b/a)\arccos{-y\over2\sqrt a}-\pi/4+{ny\over4a}\sqrt{2\sqrt a-y}\sqrt{2\sqrt a+y}]
  \\\sim&({na\over e})^{n/2}(-1)^n({\sqrt a\over\sqrt{2\sqrt a-y}\sqrt{2\sqrt a+y}})^{1/2}({\sqrt a\over-y})^{b/a}\times\exp[{ny^2\over4a}]
  \\&~~\times2\cos[(n+1/2+b/a)\arccos{-y\over2\sqrt a}-\pi/4+{ny\over4a}\sqrt{2\sqrt a-y}\sqrt{2\sqrt a+y}].
\end{align*}
This proves \eqref{IIA-3}.
\end{proof}

\subsection{Case II.B: $a<0$}
\begin{thm}\label{IIB}
Assume $d=0$ and $a<0$. Let $x=i\sqrt n y$, $A=-a>0$ and $B=-b$. As $n\to\infty$, we have for $y\in\C\cut[-2\sqrt A,2\sqrt A]$,
\begin{align}\label{IIB-1}
  \pi_n(i\sqrt n y)\sim&i^n({n\over4e})^{n/2}(y+\sqrt{y^2-4A})^n({y+\sqrt{y^2-4A}\over2\sqrt{y^2-4A}})^{1/2}({y+\sqrt{y^2-4A}\over2y})^{B/A}
  \times\exp[{ny\over4A}(y-\sqrt{y^2-4A})];
\end{align}
and for $y$ in a neighborhood of $(0,2\sqrt A)$, we have
\begin{align}\label{IIB-2}
  \pi_n(i\sqrt n y)\sim&i^n({nA\over e})^{n/2}({\sqrt A\over\sqrt{2\sqrt A-y}\sqrt{2\sqrt A+y}})^{1/2}({\sqrt A\over y})^{B/A}\times\exp[{ny^2\over4A}]
  \nonumber\\&~~\times2\cos[(n+1/2+B/A)\arccos{y\over2\sqrt A}-\pi/4-{ny\over4A}\sqrt{2\sqrt A-y}\sqrt{2\sqrt A+y}];
\end{align}
and for $y$ in a neighborhood of $(-2\sqrt A,0)$, we have
\begin{align}\label{IIB-3}
  \pi_n(i\sqrt n y)\sim&i^n({nA\over e})^{n/2}(-1)^n({\sqrt A\over\sqrt{2\sqrt A-y}\sqrt{2\sqrt A+y}})^{1/2}({\sqrt A\over-y})^{B/A}\times\exp[{ny^2\over4A}]
  \nonumber\\&~~\times2\cos[(n+1/2+B/A)\arccos{-y\over2\sqrt A}-\pi/4+{ny\over4A}\sqrt{2\sqrt A-y}\sqrt{2\sqrt A+y}].
\end{align}
\end{thm}
\begin{proof}
  The monic polynomials $p_n(z):=i^{-n}\pi_n(iz)$ satisfy the same difference equation and initial conditions of $\pi_n$ with $a$ and $b$ replaced by $A=-a$ and $B=-b$ respectively.
  Theorem \ref{IIB} follows from Theorem \ref{IIA}.
\end{proof}

\subsection{Case II.C: $a=0$}
\begin{thm}
  Assume $d=0$ and $a=0$. As $n\to\infty$, we have for $x\in\C\cut[-1,1]$,
  \begin{align}\label{IIC-1}
    \pi_n(x)\sim({x+\sqrt{x^2-1}\over2})^{n+1}{1\over\sqrt{x^2-1}};
  \end{align}
  and for $x$ in a neighborhood of $(-1,1)$, we have
  \begin{align}\label{IIC-2}
    \pi_n(x)\sim{\sin[(n+1)\arccos x]\over2^n\sqrt{1-x^2}}.
  \end{align}
  The above asymptotic formula is actually an equality.
\end{thm}
\begin{proof}
  Note that $\pi_n(x)=U_n(x)/2^n$ with $U_n(x)$ being the Chebyshev polynomials of the second kind. Furthermore, we have for $x\in\C\cut[-1,1]$,
  $$\pi_n(x)={(x+\sqrt{x^2-1})^{n+1}-(x-\sqrt{x^2-1})^{n+1}\over2^{n+1}\sqrt{x^2-1}}.$$
  It is readily seen that $\pi_n(x)\sim\Phi_n(x)$ with
  $$\Phi_n(x):=({x+\sqrt{x^2-1}\over2})^{n+1}{1\over\sqrt{x^2-1}}.$$
  This proves \eqref{IIC-1}. To be consistent, we use the argument of analytical continuity and obtain
  $$\pi_n(x)\sim\lim_{\ep\to0^+}[\Phi_n(x+i\ep)+\Phi_n(x-i\ep)]={\sin[(n+1)\arccos x]\over2^n\sqrt{1-x^2}}$$
  for $x\in(-1,1)$.
  This gives \eqref{IIC-2}. We remark that the formula \eqref{IIC-2} is actually an equality.
\end{proof}

%%%%%%%%%%%%%%%%%%%%%%%%%%%%%%%%%%%%%%%%%%%%%%%%%%%%%%%%%%%%

\end{document}